\documentclass[final,leqno]{siamltex}
 \usepackage{epsfig,subeqn,bbold,multirow}
 \title{Efficient Algorithms for Positive Semi-Definite Total Least Squares Problems, Minimum Rank Problem and Correlation Matrix Computation}

\author{Negin Bagherpour \thanks{Faculty of Mathematical Sciences, Sharif University of Technology, Tehran, Iran, (nbagherpour@mehr.sharif.ir).}
\and Nezam Mahdavi-Amiri \thanks{Faculty of Mathematical Sciences, Sharif University of Technology, Tehran, Iran, (nezamm@sharif.ir).}}

\newcommand{\tr}{\mathop{\mathrm{tr}}}

\newcommand{\rand}{\mathop{\mathrm{rand}}}

\begin{document}
\maketitle
\begin{abstract}
We have recently presented a method to solve an overdetermined linear system of equations with multiple right hand side vectors, where the unknown matrix is to be symmetric and positive definite. The coefficient and the right hand side matrices are respectively named data and target matrices. A more complicated problem is encountered when the unknown matrix is to be positive semi-definite. The problem arises in estimating the compliance matrix to model deformable structures and approximating correlation and covariance matrices in financial modeling. Several methods have been proposed for solving such problems assuming that the data matrix is unrealistically error free. Here, considering error in measured data and target matrices, we propose a new approach to solve a positive semi-definite constrained total least squares problem. We first consider solving the problem when the rank of the unknown matrix is known, by defining a new error formulation for the positive semi-definite total least squares problem. Minimization of our newly defined error consists of an optimization problem on the Stiefel manifold. To solve the optimization problem, in each iteration a linear operator subproblem arises for which we propose three different iterative methods. We prove quadratic convergence of our proposed approach. We then describe how to generalize our proposed method to solve the general positive semi-definite total least squares problem. We further apply the proposed approach to solve the minimum rank problem and the problem of computing correlation matrix. Comparative numerical results show the efficiency of our proposed algorithms. In solving positive semi-definite total least squares problems, we find that in most cases the linear operator equation is solved faster and turns to be more accurate using the GMRES method. Also, comparison of the results obtained by our algorithm with the ones due to two other methods, the interior point method and a MATLAB routine for solving a quadratic programming problem with semi-definite constraint based on a path following algorithm, confirms the efficiency of our approach. Numerical test results also show that our approach for computing a correlation matrix leads to smaller standard deviations of error in the target matrix. Finally, the Dolan-Mor\'{e} performance profiles are shown to summarize our comparative study.\\
\end{abstract}

\markboth{NEGIN BAGHERPOUR AND NEZAM MAHDAVI-AMIRI}{Solving Positive Semidefinite Total Least Squares Problems}

\section{Introduction}
 In several physical problems, such as estimation of the mass inertia matrix in the design of controllers for solid structures and robots, an overdetermined linear system of equations with multiple right hand side vectors arises with the constraint that the unknown matrix be symmetric and positive definite; see, e.g., \cite{4,7,23}. A method for solving such a problem has been proposed in \cite{24}. There are also physical contexts, such as modeling a deformable structure \cite{5,18} and computing the correlation matrix in finance or insurance/reinsurance industries \cite{9,12,22,27}, where a symmetric positive semi-definite solution of an over determined linear system of equations needs to be computed or equivalently the problem
\begin{equation}\label{1}
  DX \simeq T
\end{equation}

needs to be solved, where $D,T \in {\mathbb{R}}^{m \times n}$, with 
$m\geq n$, are given and $X \in {\mathbb{R}}^{n \times n}$, a symmetric positive semi-definite matrix, is to be computed as a solution. In some special applications, the data matrix $D$ has a simple structure, which may be taken into consideration for efficiently organized computations. Computing the correlation matrix in finance is such an example where the data matrix is the identity matrix; see, e.g., \cite{27}.\\

Unlike the positive definite total least squares problem, here the unknown matrix is singular and thus our previously defined error formulation in \cite{24} is no more applicable. We need to formulate the error in the measured data and target matrices as a function of the unknown matrix but not its inverse.\\

A number of least squares formulations have been proposed for the physical problems, which may be classified as ordinary and total least squares problems. Unlike the ordinary formulation, in a total least squares formulation both data and target matrices are assumed to contain error. Also, single or multiple right hand sides may arise. In \cite{24}, ordinary and total least squares formulations with single or multiple right hand sides have been considered. For detailed analysis of total least squares, see \cite{2,8,17}.\\

Here, we consider an specific case of the total least squares problem with multiple right hand side vectors. Our goal is to compute a symmetric positive semi-definite solution $X \in {\mathbb{R}}^{n \times n}$ of the overdetermined system of equations $DX\simeq T$, where both matrices $D$ and $T$ may contain error. Several approaches have been proposed for this problem, commonly considering the ordinary least squares formulation and minimizing the error ${\| \Delta T\|}_F$ over all $n \times n$ symmetric positive semi-definite matrices, where ${\|.\|}_F$ is the Frobenious norm.
Larson \cite{6} discussed a method for computing a symmetric solution to an overdetermined linear system of equations based on solving the corresponding normal system of equations. Krislock \cite{18} proposed an interior point method for solving a variety of least squares problems with positive definiteness constraint. Woodgate \cite{15} described a new algorithm for solving a similar problem in which a symmetric positive semi-definite matrix $P$ is computed to minimize $\|F-PG\|$, with known $F$ and $G$. In \cite{14}, Toh introduced a path following algorithm for solving a positive semi-definite quadratic optimization problem. Later in 2009, he posted a MATLAB package for solving such a problem; see \cite{30}. Hu \cite{3} gave a quadratic programming approach to solve a least squares problem with a symmetric positive definite unknown matrix. In his method, the upper and lower bounds for the entries of the target matrix can be given as extra constraints. In real measurements, however, both the data and target matrices may contain error. Thus, to be practical, a total least squares formulation seems to be appropriate. Here, we define a new error function to consider error in both data and target matrices and propose an iterative algorithm to minimize the defined error.\\

If the goal is to compute the correlation matrix, the mathematical problem is a little different. Computing the correlation matrix is very important in financial modeling. It is applicable for example in obtaining a quadratic model for an economical system and even in reverse engineering for extreme scenario stress testing \cite{25}. In this case, the data matrix is the identity and a large number of linear constraints are to be satisfied. Sun \cite{12} presented an algorithm for computing the correlation matrix. Rebonato \cite{9} and Werner \cite{22} also discussed solving the same problem. We will see later that the minimum rank problem can also be solved by applying our proposed algorithm. This problem appears in the literature in diverse areas including system identification and control, Euclidean embedding, and collaborative filtering; see \cite{10,16,31}. In a minimum rank problem, the goal is to find a positive semi-definite solution with the minimum possible rank to an overdetermined linear system of equations.\\

The remainder of our work is organized as follows. In Section 2, we define a new error function for solving a positive semi-definite total least squares problem with a fixed rank. A method for solving the resulting optimization problem is presented in Section 3. Also, a discussion on solving the positive semi-definite total least squares problem (with arbitrary rank) is given in Section 3. In Section 4, we introduce two slightly different problems and discuss how to solve them based on the proposed method in Section 3. These two problems are: the minimum rank problem and computing the correlation matrix. Comparative computational results are given in Section 5. Section 6 gives our conclusion.
\section{Problem Formulation}
Available methods for solving a positive semi-definite least squares problem consider an ordinary least squares formulation; see, e.g., \cite{14,18}. A practically useful total error formulation was introduced in \cite{24} for a positive definite total least squares problem. Based on this formulation, the solution of the optimization problem
 \begin{equation}\label{2}
  \min\limits_{X\succ 0} \tr(DX-T)^T (D-TX^{-1})
\end{equation}
is a solution of a corresponding positive definite total least squares problem, where $X$ is symmetric and by $X\succ 0$, we mean $X$ is positive definite. The error formulation in \cite{24} not being suitable here, we first motivate and present a new error formulation for the positive semi-definite total least squares case.\\

In (\ref{2}), the entries of $D-TX^{-1}$ and $DX-T$ represent the errors in $D$ and $T$, respectively. Here, we need to represent the error in $D$ independent of $X^{-1}$. Before discussing how to solve the positive semi-definite total least squares problem, we consider the newly noted problem, positive semi-definite total least squares problem with a given rank, $r$, of the unknown matrix (R$r$-PSDTLS). In Section 3, we outline an algorithm for solving R$r$-PSDTLS and discuss how to solve the positive semi-definite total least squares problem applying the proposed algorithm.\\

The error in $D$ is supposed to be the difference between the real value of $D$ and the predicted value for $D$ obtained by $DX\simeq T$. To compute the predicted value for $D$, we use the general least squares solution of the system $XD^T\simeq T^T$. Considering the block form $D^T=\left(
                                              \begin{array}{c}
                                                 d_1^T \\
                                                 \vdots \\
                                                 d_n^T
                                               \end{array}
                                             \right)$ and $T^T=\left(
                                               \begin{array}{c}
                                                 t_1^T \\
                                                 \vdots \\
                                                 t_n^T
                                               \end{array}
                                             \right)$, where $d_i,t_i \in {\mathbb{R}}^{m}$, for $i=1,\cdots,n$, we have $Xd_i=t_i$, for $i=1,\cdots,n$. The general solution to such a linear system has the form
                                             \begin{equation}\label{3}
                                             d_i=X^{\dag} t_i+n_i,
                                             \end{equation}
                                             where $X^{\dag}$ is the pseudo-inverse of $X$ and $n_i$ is an arbitrary vector in the null space of $X$ \cite{21}. A straight choice for $n_i$ is $n_i=0$ which results in $d_i=X^{\dag} t_i$ and $\Delta D=D-TX^{\dag}$. We later consider the suitable choices for $n_i$ which minimizes ${\|\Delta D\|}_F$. To compute $d_i$ from (\ref{3}), the spectral decomposition can be applied.\\\\

\textbf{Result.} \textrm{(Spectral decomposition) \cite{21}
All eigenvalues of a symmetric matrix, $A\in {\mathbb{R}}^{n\times n}$, are real and there are $n$ mutually orthonormal vectors representing the corresponding eigenvectors. Thus, there exist an orthonormal matrix $\bar{U}$ with columns being the eigenvectors of $A$ and a diagonal matrix $\bar{D}$ containing the eigenvalues such that $A=\bar{U}\bar{D}{\bar{U}}^T$. Suppose that $D$ and $U$ are respectively formed by ordering the diagonal elements of $\bar{D}$ non-increasingly and rearranging the corresponding columns of $\bar{U}$. The ordered spectral decomposition of $A$ has the form $A=UDU^T$. Also, if $A$ is positive semi-definite with $\rank(A)=r$, then $r$ of its eigenvalues are positive and the rest are zero, and thus we can set $D=\left(
                                               \begin{array}{cc}
                                                 S^2 & 0 \\
                                                 0 & 0
                                               \end{array}
                                             \right)$, where $S^2\in {\mathbb{R}}^{r\times r}$ is diagonal and nonsingular. Hence, the ordered spectral decomposition of a symmetric positive semi-definite matrix $A$ is $A=U\left(
                                               \begin{array}{cc}
                                                 S^2 & 0 \\
                                                 0 & 0
                                               \end{array}
                                             \right)U^T$, with $U^TU=UU^T=I$. Considering the block form $U=\left( U_r\hspace{0.3cm} U_{n-r}\right)$, we get $A=U_rS^2U_r^T$. Moreover, the columns of matrices $U_r$ and $U_{n-r}$ respectively form a basis for the range and null space of $A$.\\\\}

Now, making use of the spectral decomposition of $A$ in (\ref{3}), we have
                                             \[d_i=X^{\dag} t_i+U_{n-r}z_i,\]
                                             where $z_i \in {\mathbb{R}}^{n-r}$, for $i=1,...,m$, are arbitrary vectors, and
                                             \[D^T=X^{\dag} T^T+U_{n-r}\left(
                                               \begin{array}{ccc}
                                                 z_1 & \cdots & z_m
                                               \end{array}
                                             \right).\]
                                             Thus, the predicted value for $D$ is
                                             \begin{equation}\label{5}
                                             D_p=TX^{\dag}+ZU_{n-r}^T,
                                             \end{equation}
                                             with $Z\in {\mathbb{R}}^{m\times (n-r)}$, arbitrary, and the error in $D$ is equal to $\Delta D=D-(TX^{\dag}+ZU_{n-r}^T)$. A reasonable choice for $Z$ in this formulation would be the one minimizing the norm of $\Delta D$, which is the solution of the optimization problem
                                             \begin{equation}\label{6}
                                             \min\limits_{Z} {\|F-ZU_{n-r}^T\|}_F^2,
                                             \end{equation}
                                             where $F=D-TX^{\dag}$ and ${\|.\|}_F$ is the Frobenius norm. Solving (\ref{6}) results in \cite{21}
                                             \begin{equation}\label{7}
                                             Z^{\ast}=FU_{n-r}=(D-TX^{\dag})U_{n-r}.
                                             \end{equation}
                                             Substituting (\ref{7}) in (\ref{5}), we get
\begin{eqnarray}
\Delta D &=& D-(TX^{\dag}+Z^{\ast}U_{n-r}^T) = (D-TX^{\dag})- (D-TX^{\dag})U_{n-r}U_{n-r}^T\nonumber \\
 &=& (D-TX^{\dag})(I-U_{n-r}U_{n-r}^T). \label{8}
\end{eqnarray}
Using $I-U_{n-r}U_{n-r}^T=U_rU_r^T$ along with (\ref{8}), we get
\[\Delta D = (D-TX^{\dag})U_rU_r^T.\]
Based on the above discussion, $\Delta D =  (D-TX^{\dag})U_rU_r^T$ and $\Delta T= DX-T$ represent the error in $D$ and $T$ respectively. Thus, to solve a rank $r$ positive semi-definite total least squares problem, it is appropriate to minimize the error
\[E={\sum}_{i=1}^m{\sum}_{j=1}^n{\Delta D}_{ij}{\Delta T}_{ij}=\tr({\Delta T}^T{\Delta D}),\]
with $\tr(\cdot)$ standing for trace of a matrix. Consequently, the optimization problem
\begin{eqnarray}
&\min& \tr({\Delta T}^T{\Delta D})\label{10} \\
&s.t.& \nonumber\\
&&X\succeq 0\nonumber\\
&&\rank(X)=r\nonumber
\end{eqnarray}
needs to be solved, where $X$ is symmetric and by $X\succeq 0$, we mean $X$ is positive semi-definite. We can rewrite the optimization problem using the spectral decomposition of $X$ and substituting $X$ and $X^{\dag}$ by $U_rS^2U_r^T$ and $U_rS^{-2}U_r^T$ respectively. Considering well-known properties of the trace operator \cite{21} and the above formulation for $X$ and $X^{\dag}$, we get
\begin{eqnarray}
E&=&\tr({\Delta T}^T{\Delta D})=\tr {(DX-T)}^T(D-TX^{\dag})U_rU_r^T \nonumber\\
 &=&\tr (U_rS^2U_r^TD^T-T^T)(D-TU_rS^{-2}U_r^T)U_rU_r^T\nonumber\\
 &=&\tr (S^2U_r^TD^T-U_r^TT^T)(DU_r-TU_rS^{-2})\nonumber\\
 &=&\tr (U_r^TD^TDU_rS^2-U_r^TT^TDU_r-U_r^TT^TDU_r+U_r^TT^TTU_rS^{-2}).\label{11}
\end{eqnarray}
Letting $A=D^TD$, $C=D^TT+T^TD$ and $B=T^TT$, problem (\ref{10}) is then equivalent to
\begin{equation}\label{12}
\min\limits_{Y,S}\tr (Y^TAYS^2-Y^TCY+Y^TBYS^{-2}),
\end{equation}
where $Y\in {\mathbb{R}}^{n\times r}$ satisfies $Y^TY=I$ and $S\in {\mathbb{R}}^{r\times r}$ is a nonsingular diagonal matrix.\\

The Lagrangian function corresponding to the constrained optimization problem
\begin{eqnarray}\nonumber
&\min& f(X)\nonumber\\
&s.t.& g(X)=0\nonumber
\end{eqnarray}
is $L(X,\lambda)=f(X)+{\lambda}^T g(X)$, where the Lagrangian coefficient vector ${\lambda}$ corresponds to the constraint vector $g(X)=0$. Necessary conditions for a solution, known as Karush-Kahn-Ticker conditions, is ${\nabla}_XL(X,\lambda)=0$ as well as ${\nabla}_{\lambda}L(X,\lambda)=0$, which gives $g(X)=0$; for KKT conditions, see \cite{19}. Thus, $L(Y,S,\Omega)=\tr (Y^TAYS^2-Y^TCY+Y^TBYS^{-2}+\Omega(Y^TY-I))$ is the Lagrangian function for the optimization problem (\ref{12}).\\

\begin{lemma}
An appropriate characteristic of the error formulation proposed by
\begin{equation}\label{66}
E=\tr (Y^TAYS^2-Y^TCY+Y^TBYS^{-2}),
\end{equation}
is that its value is nonnegative and it is equal to zero if and only if $\Delta D=0$.
\end{lemma}
\begin{proof}
It is clear that if $\Delta D=0$, then $E=\tr ({\Delta T}^T \Delta D)=0$. Assuming $E=0$, from (\ref{66}) we have
\[E=\tr {(DYS-TYS^{-1})}^T(DYS-TYS^{-1})=0,\]
which holds if and only if
\[DYS=TYS^{-1}.\]
Multiplying both sides from right by $SY^T$, we get
\[DYS^2Y^T=TYY^T,\]
or equivalently
\begin{equation}\label{67}
DX=TX^{\dag}X.
\end{equation}
If (\ref{67}) is satisfied, then we have $(D-TX^{\dag})X=0$; hence, $\Delta D=(D-TX^{\dag})XX^{\dag}=0$.
\end{proof}

\section{Mathematical Solution}
Here, we discuss how to solve (\ref{12}). We also study the computational complexity and the convergence properties of our proposed algorithm.
\subsection{\textbf{Solving Rank $r$ Positive Semi-definite Total Least
Squares Problem}}
We are to propose an algorithm for solving (\ref{12}). More precisely, a nonsingular diagonal matrix $S=diag(s_1,\cdots,s_r)$ and a matrix $Y\in V_r({\mathbb{R}}^n)$ need to be computed to minimize
\[E=f(Y,S)=\tr (Y^TAYS^2-Y^TCY+Y^TBYS^{-2}),\]
where $V_r({\mathbb{R}}^n)$ is the Stiefel manifold \cite{1}:
\[V_r({\mathbb{R}}^n)=\{ A\in {\mathbb{R}}^{n\times r}; A^TA=I\}.\]
In the following lemma, we show that the optimization problem (\ref{12}) is strictly convex under a weak assumption on the data and target matrices. We also make use of the well-known properties of convexity to propose our algorithm.\\\\
\begin{lemma}\label{55}
The function $f(Y,S)=\tr (Y^TAYS^2-Y^TCY+Y^TBYS^{-2})$ is always convex and it is strictly convex on the set $\{(Y,S)|Y\in {\mathbb{R}}^{n \times r}, S\in F\}$, where $F=\{S=diag(s_1,\ldots,s_r)| \rank(s_i^2D-T)=n \hspace{0.2cm} for \hspace{0.2cm} i=1,\ldots,r\}$.
\end{lemma}
\begin{proof}
The key point of the proof is to reformulate $f(Y,S)$ as
\begin{eqnarray}\nonumber
f(Y,S)&=&\tr (Y^TAYS^2-Y^TCY+Y^TBYS^{-2})\\
&=&{\sum}_{i=1}^{r}s_i^2y_i^TAy_i-y_i^TCy_i+\frac{1}{s_i^2}y_i^TBy_i\nonumber\\
&=&{\sum}_{i=1}^{r}y_i^T{(s_iD-\frac{1}{s_i}T)}^T(s_iD-\frac{1}{s_i}T)y_i.\nonumber
\end{eqnarray}
Thus, $f(Y,S)$ is always convex and it is strictly convex if and only if $H_i={(s_iD-\frac{1}{s_i}T)}^T$\\$(s_iD-\frac{1}{s_i}T)$, for $i=1,\ldots,r$, is positive definite which holds if and only if $s_iD-\frac{1}{s_i}T$, for $i=1,\ldots,r$, has full column rank.
\qquad
\end{proof}

\vspace{0.3cm}

\textbf{Note.} Since the function $f(Y,S)$ is strictly convex and the set $\{Y| Y^TY=I\}$ is convex, a point $(Y^{\ast},S^{\ast})$ satisfying the Karush--Kuhn--Tucker (KKT) optimality conditions is the unique solution of problem (\ref{12}); for KKT conditions, see \cite{19}.\\\\
Next, in the following theorem we derive the KKT optimality conditions for problem (\ref{12}).\\
\begin{theorem}\label{14}
If $(Y,S)$ is the solution of problem (\ref{12}), then it satisfies
\begin{equation}\label{27}
s_i={(\frac{{y_i}^TBy_i}{{y_i}^TAy_i})}^{\frac{1}{4}},
\end{equation}
where $y_i$ is the $i$th column of $Y$.\\
\end{theorem}

\begin{proof} The KKT necessary conditions for (\ref{12}) are obtained by setting $\nabla L=0$. Thus, if $(Y,S)$ forms a solution for (\ref{12}) with $S=diag(s_1, \cdots, s_r)$, then we must have $\frac{\partial L}{\partial s_i}=0$, for $i=1,\cdots,r$. To simplify the computation of $\frac{\partial L}{\partial s_i}$, we can reformulate $L(Y,S,\Omega)$ using the definition of the trace operator. Let $G=Y^TAYS^2$ and $H=Y^TBYS^{-2}$. We have
\[\tr(G)={\sum}_{i=1}^r G_{ii}={\sum}_{i=1}^rs_i^2{y_i}^TAy_i\]
and
\[\tr(H)={\sum}_{i=1}^r{\Delta H}_{ii}={\sum}_{i=1}^r\frac{1}{s_i^2}{y_i}^TBy_i.\]
So, $L(Y,S,\Omega)$ is equal to
\begin{eqnarray}
L(Y,S,\Omega)&=&{\sum}_{i=1}^rs_i^2{y_i}^TAy_i+{\sum}_{i=1}^r\frac{1}{s_i^2}{y_i}^TBy_i\label{15}\\
&-&\tr (Y^TCY-\Omega(Y^TY-I)).\nonumber
\end{eqnarray}
Now, from (\ref{15}) we have
\begin{center}
$\frac{\partial L}{\partial s_i}=2s_i({y_i}^TAy_i)-\frac{2}{s_i^3}({y_i}^TBy_i)=0,$
\end{center}
and $s_i$ can be computed by $s_i={(\frac{{y_i}^TBy_i}{{y_i}^TAy_i})}^{\frac{1}{4}}$.
\qquad
\end{proof}

\vspace{0.4cm}

Considering the above discussion, in each iteration of our proposed algorithm we need to\\
(1) compute one term of the sequence $Y_i \in V_r({\mathbb{R}}^n)$ converging to the minimizer of the\\$\indent$error $E=f(Y,S)$, and\\
(2) compute the diagonal elements of $S$ from (\ref{27}).\\

Edelman \cite{1} introduced two methods for solving optimization problems on Stiefel manifolds: the Newton method and conjugate gradient method on the Stiefel manifold. Here, we adaptively use the Newton approach to develop an algorithm for solving (\ref{12}). In each iteration of our proposed algorithm, a Newton step is computed from the Newton method on the Stiefel manifold and then the diagonal elements of $S$, $s_i$, are updated by (\ref{27}). We show in Section 3.2 that our proposed algorithm converges to the unique solution of (\ref{12}) at least quadratically. Also, we discuss its computational complexity in Section 3.2.\\\\
We are now ready to outline the steps of our proposed algorithm.\\

%\begin{algorithm}[H]
\textbf{Algorithm 1.} Solving rank $r$ positive semi-definite total least squares problem using the Newton method on Stiefel manifold (R$r$-PSDTLS).\\
%\begin{algorithmic}[1]
%\Procedure {R$r$-PSDTLS}{$D$, $T$, $r$, $\epsilon$, $\delta$}
- $\epsilon$ and $\delta$ are upper bounds for relative and absolute error, respectively taken to be close to the machine (or user's) unit roundoff error and machine (or user's) zero.\\
(1) Let $A=D^TD$, $B=T^TT$ and $C=D^TT+T^TD$.\\
(2) Choose $Y$ such that $Y^TY=I$.\\
Repeat\\
(3.1) Let $s_i={(\frac{{y_i}^TBy_i}{{y_i}^TAy_i})}^{\frac{1}{4}}$.\\
(3.2) Compute the $n\times r$ matrix $F_Y$ such that ${F_Y}_{ij}=\partial E/\partial Y_{ij}$ and let
\[G=F_Y-YF_Y^TY.\]
(3.3) To compute $\Delta$, solve the linear system of equations
\[F_{YY}(\Delta)-Yskew(F_Y^T\Delta)-skew(\Delta F_Y^T)Y-\frac{1}{2}(I-YY^T) \Delta Y^TF_Y^T=-G,\]
$\hspace{1cm}$where 
\[F_{YY}(\Delta)=A\Delta S^2-YS^2{\Delta}^TAY-C\Delta+Y{\Delta}^TCY+B\Delta S^{-2}-YS^{-2}{\Delta}^TBY\]
$\hspace{1cm}$and $skew(X)=\frac{X-X^T}{2}$.\\
(3.4) Move from $Y$ in direction $\Delta$ to $\bar{Y}$ using $\bar{Y}=YM+QN$, where
\[(I-YY^T)\Delta=QR\]
$\hspace{1cm}$is the compact $QR$ factorization of $(I-YY^T)\Delta$, and $M$ and $N$ are:
\[\left(
                                               \begin{array}{c}
                                                 M \\
                                                 N
                                               \end{array}
                                             \right)=exp(\left(
                                               \begin{array}{cc}
                                                 K & -R^T \\
                                                 R & 0
                                               \end{array}
                                             \right)),\]
$\hspace{1cm}$with $K=Y^T\Delta$.\\
(3.5) Compute $Error=\|\bar{Y}-Y\|$. Let $Y=\bar{Y}$.\\
until {$Error\leq \epsilon \|Y\|+\delta$.}\\
(4) Let $X=YS^2Y^T$ and $E=\tr (Y^TAYS^2-Y^TCY+Y^TBYS^{-2})$.\\\\
%\end{algorithmic}
%\end{algorithm}
\textbf{Note.} The linear equation
\begin{equation}\label{17}
F_{YY}(\Delta)-Yskew(F_Y^T\Delta)-skew(\Delta F_Y^T)Y-\frac{1}{2}(I-YY^T) \Delta Y^TF_Y^T=-G
\end{equation}
may be solved by various methods including conjugate gradient and GMRES \cite{20}. Another possible method is to convert the linear operator appearing on the left side of (\ref{17}) to an $nr\times nr$ linear system of equations. In Section 5, we present the numerical results obtained by using these three methods and compare the respective obtained accuracies and computing times.\\
\subsection{\textbf{Solving Semi-definite Total Least Squares Problem}} In Section 3.1, we outlined Algorithm 1 to solve the rank $r$ positive semi-definite total least
squares problem. Here, we discuss how to solve the general positive semi-definite total least
squares problem.\\

A positive semi-definite solution for the overdetermined linear system of equations $DX\simeq T$, whose rank is not known, needs to be computed.
This problem arises for example in estimation of compliance matrix of a deformable structure \cite{18}.
To solve this problem, we can apply PSDTLS for possible values of  $r=1,\cdots,n$, compute the corresponding solutions $X_r=Y^TS^2Y$, and identify the one
 minimizing $E=\tr (Y^TAYS^2-Y^TCY+Y^TBYS^{-2})$. We will refer to this approach as PSDTLS. In Section 5.1, we report some numerical results to compare PSDTLS by two existing methods. Although our proposed method (PSDTLS) computes the minimizer of $E$ for each value of $r=1,\cdots,n$ and then finds the optimal solution among $X_r$, for $r=1,\cdots,n$, it takes less time to solve the problem than two other proposed methods in the literature.

\subsection{\textbf{Convergence Properties and Computing Cost}}
Here, we discuss convergence properties of R$r$-PSDTLS. We cite a theorem to be used to establish the local quadratic convergence of R$r$-PSDTLS to the unique solution of (\ref{12}). We also show that the computational complexity of every iteration of our proposed algorithm is
\[N=2mn^2+n^3r^2+n^2(r^2+r)+2n^2r.\]
Moreover, we provide an upper bound for the computational complexity of our proposed approach for solving the positive semi-definite total least squares problem, PSDTLS.
\begin{theorem} \label{28}
(\cite{11}) Newton's method \cite{1} applied to the function $f(Y)=\tr (Y^TQYN)$ on the Stiefel manifold
\[V_r({\mathbb{R}}^n)=\{ Y\in {\mathbb{R}}^{n\times r}; Y^TY=I\},\]
locally converges to the unique solution of
\begin{eqnarray}
\min F(Y), \nonumber \\
Y^TY=I,\nonumber\\
\end{eqnarray}
at least quadratically.\\
\end{theorem}
\begin{proof}
See \cite{11}.
\end{proof}\\

\begin{lemma}
Algorithm 1 converges locally to the unique solution of problem (\ref{12}) at least quadratically.
\end{lemma}

\begin{proof}
In Algorithm 1, we have two main computations: applying Newton's approach on Stiefel manifold to update $Y$ and updating the scalars $s_i$ using (\ref{27}). The rate of convergence is not affected by (\ref{27}) and it is governed by Newton's approach. Thus, considering Theorem \ref{28}, R$r$-PSDTLS converges at least quadratically to the unique solution of (\ref{12}).
\qquad
\end{proof}

$\vspace{0.4cm}$

$\hspace{-0.6cm}$\textbf{Computing Cost: Rank $r$ Positive Semi-definite Total Least Squares Problem.}
 The computational complexity of one iteration of R$r$-PSDTLS is given in Table \ref{t1}. The first, second and third columns respectively give the computational complexities of solving the linear problem (\ref{17}) using conjugate gradient method in the operator form (CG-O), GMRES in the operator form (GMRES-O) and conjugate gradient method after converting (\ref{17}) into a linear system of equations (CG-L); for details, see \cite{20}.
\begin{table}[htbp]
\caption{Computational complexities for one iteration using different approaches.}
\label{t1}
\begin{center}\footnotesize
%% \multirow{ 3}{*}{PSDLS} \multicolumn{6}{c}{$\ast$: Out of memory.}
%\renewcommand{\arraystretch}{1.3}
\begin{tabular}{|c|c|c|c|}\hline
{\multirow{2}{*}{Computation}} & \multicolumn{3}{c|}{Time complexity} \\ \cline{2-4}
 & CG-O & GMRES-O & CG-L\\
    \hline
$s_i$ & $2n^2r$ & $2n^2r$ & $2n^2r$\\
$G$ & $n^2(r^2+r)$ & $n^2(r^2+r)$ & $n^2(r^2+r)$\\
Solving (\ref{17}) & $n^3r^2$ & $n^3r^2$ & $n^3r^3$\\
    \hline
\multirow{2}{*}{Total complexity} & $n^3r^2+n^2(r^2+r)$ & $n^3r^2+n^2(r^2+r)$ & $n^3r^3+n^2(r^2+r)$\\
 {} & $+2n^2r$ & $+2n^2r$ & $+2n^2r$\\
    \hline
\end{tabular}
\end{center}
\end{table}
\normalsize
Also, the complexity of computing $A$, $B$ and $C$ in step (1) of R$r$-PSDTLS is $2mn^2$. Thus, the total computational complexity of performing one iteration of R$r$-PSDTLS is
\[N_r=2mn^2+n^3r^2+n^2(r^2+r)+2n^2r.\]
As shown in Table 3.1, the computational complexity of CG-L is approximately $r$ times greater than the ones due to CG-O and GMRES-O.\\\\
\textbf{Computing Cost: Positive Semi-definite Total Least Squares Problem.
}The computational cost for solving the positive semi-definite total least squares problem is equal to the sum of the computational costs for solving all the rank $r$ positive semi-definite total least squares problems. The GMRES algorithm for solving the $n\times r$ operator equation (\ref{17}) is terminated after at most $n$ iterations \cite{20}; hence, the maximum computational cost would be
\[N=2mn^2+{\sum}_{r=1}^nn(n^3r^2+n^2(r^2+r)+2n^2r)\simeq 2mn^2+\frac{n^7}{3}+\frac{n^6}{3}+n^5+\frac{n^4}{2}.\]

\subsection{\textbf{An Special Case}}
In some applications like rank one signal recovery, a positive semi-definite rank 1 least squares problem needs to be solved; see, e.g., \cite{26}. The following lemma is concerned with the special case $r=1$. \\
\begin{lemma} If $r=1$, then problem (\ref{12}) can be converted to a quadratic eigenvalue problem.\\
\end{lemma}

\begin{proof}
Reformulating the Lagrangian function for the optimization problem (\ref{12}), for the case $r=1$, we have
\[L(y,s,\Omega)=s^2y^TAy+\frac{1}{s^2}y^TBy-y^TCy-\Omega(y^Ty-1)).\]
Let $u=sy$ and $v=\frac{1}{s}y$. Thus,
\[L(u,v,\Omega)=u^TAu+v^TBv-u^TCv-\Omega(u^Tv-1)).\]
The KKT necessary optimality conditions lead to
\begin{eqnarray}
2Au-Cv-\Omega v=0,\label{18}\\
2Bv-Cu-\Omega u=0,\nonumber\\
u^Tv=1.\nonumber
\end{eqnarray}
If $D$ and $T$ have full ranks, then it can be concluded from (\ref{18}) that
\begin{eqnarray}
(2B-\frac{1}{2}(C+\Omega I)A^{-1}(C+\Omega I))v=0,\label{19}\\
u=\frac{1}{2}A^{-1}(C+\Omega I)v.\nonumber
\end{eqnarray}
Note that (\ref{19}) is a quadratic eigenvalue problem which may be solved by various methods \cite{13,28}. This approach will be referred as R$1$-PSDTLS.
\end{proof}

Next, we point out two mathematical problems and describe how to solve them using R$r$-PSDTLS.
\section{Two Problems and Their Solutions}
Two slightly different problems also arise in some context. Here, we describe these problems and show how to solve these problems making use of Algorithm 1.\\\\
($i$) Positive semi-definite total minimum rank problem:
\begin{eqnarray}
&\min& \rank(X)\nonumber\\
&s.t.&\\
&&\tr({\Delta D}^T\Delta T)<e\nonumber\\
&& X\succeq 0.\nonumber
\end{eqnarray}
This problem arises in different contexts such as system identification and control, Euclidean embedding and collaborative filtering \cite{10}. Both ordinary and total least squares formulations have been considered for solving the problem \cite{10,16}. In an ordinary formulation, the minimum possible rank of a positive semi-definite matrix $X$ needs to be computed so that the ordinary least squares error, $\|DX-T\|$, is less than an error bound, $e$. In a total formulation, however, the goal is to minimize $\rank (X)$, where $X$ is a positive semi-definite matrix satisfying $\|[\Delta D,\Delta T]\|<e$. To solve this problem using R$r$-PSDTLS, our proposed total error $E=\tr({\Delta D}^T\Delta T)$ needs to satisfy the constraint $E<e$. We start from the smallest possible rank, $r=1$, and solve the corresponding positive semi-definite total least squares problem. If the inequality $E<e$ holds for the computed $X_r$, then we stop; otherwise, we increase $r$ by one and continue iteratively until satisfying the inequality $E<e$. The iteration might be performed at most $n$ times. If after $n$ iterations, none of the matrices $X_r$, $r=1,\cdots n$, satisfies the inequality $E<e$, then we consider $X_r$ with the smallest corresponding value of $E$ as the solution of the minimum rank problem. We summarize the discussion above in the following algorithm.

%\begin{algorithm}
\textbf{Algorithm 2.} Solving minimum rank positive semi-definite total least squares (MRPSDTLS) problem.\\
%\begin{algorithmic}[1]
%\Procedure {MRPSDTLS}{$D$, $T$, $e$}
(1) Let $A=D^TD$, $B=T^TT$ and $C=D^TT+T^TD$.\\
(2) Let $r=1$.\\
(3) Apply R$r$-PSDTLS with the input arguments $D$, $T$ and $r$ and compute $X$ and $E$.\\
If {$E<e$}\\
then let $X^{\ast}=X$ and stop\\
Else\\
Let $r=r+1$ and go to (3).\\
EndIf\\
%\end{algorithmic}
%\end{algorithm}

A significant characteristic of our proposed approach for solving the minimum rank problem is that for each rank $r$ the minimum value of $E$ is determined applying Algorithm 1. Hence, to solve the minimum rank problem, it is sufficient to find the minimum value of $r$
satisfying the inequality constraint.\\
($ii$) Computing the correlation matrix: This problem is an special case of the positive semi-definite total least squares problem. Computing a correlation matrix is equivalent to finding a positive semi-definite matrix $X$ to satisfy
\begin{eqnarray}
X\simeq C, \nonumber\\
PX\simeq Q. \nonumber
\end{eqnarray}
The linear constraints $X\simeq C$ and $PX\simeq Q$ can be replaced by the overdetermined system of equations
\begin{equation}\label{29}
\left(
    \begin{array}{c}
      I \\
      P \\
    \end{array}
  \right)X=\left(
             \begin{array}{c}
               C \\
               Q \\
             \end{array}
           \right).
\end{equation}
To solve the overdetermined system of equations (\ref{29}), both ordinary and total formulations have been considered \cite{9,12,22,27}. We note that (\ref{29}) is an special case of positive semi-definite total least squares problem with data and target matrices $D=\left(
  \begin{array}{c}
   I \\
   P \\
  \end{array}
\right)$ and $T=\left(
          \begin{array}{c}
            T \\
             Q \\
          \end{array}
        \right)$, where $I$ is the $n\times n$ identity matrix, $P, Q\in {\mathbb{R}}^{m \times n}$ and $C\in  {\mathbb{R}}^{n \times n}$ are arbitrary. Here, we make use of the PSDTLS approach described in Section 3.2 for solving (\ref{29}).\\

Next, we report some numerical results. We provide comparison of our proposed algorithms and some existing methods on randomly generated test problems.
\section{Numerical Results}
We made use of MATLAB 2012b in a Windows 7 machine with a 3.2 GHz CPU and a 4 GB RAM to implement our proposed algorithms and other methods. We generated random test problems with random data and target matrices. These random matrices were produced using the rand command in MATLAB. The command $R=\rand(m,n)$ generates an $m \times n$ matrix $R$, with uniformly distributed random entries in the interval $[0,1]$. Using the linear transformation $Rab=(b-a)*R+a$, the matrix $Rab$ with random entries in the interval $[a,b]$ is generated.\\

The numerical results are presented in two parts. Section 5.1 represents the numerical results corresponding to the rank $r$ positive semi-definite least squares problem and the positive semi-definite total least squares problem. The numerical results for the two problems mentioned in Section 4 are reported in Section 5.2. The abbreviated names of problems and the corresponding inputs and outputs are listed in tables \ref{t2} and \ref{t3}.
\begin{table}[htbp]
\caption{Abbreviated names}
\label{t2}
\begin{center} \footnotesize
\begin{tabular}{|c|c|} \hline
Problem name & Abbreviation\\ \hline
Rank $r$ positive semi-definite least squares & R$r$-PSDLS \\
Positive semi-definite least squares & PSDLS \\
Minimum rank & MR\\
Correlation matrix & CM\\ \hline
\end{tabular}
\end{center}

\end{table}
\normalsize
\newpage
\begin{table}[htbp]
\caption{Inputs and outputs}
\label{t3}
\begin{center} \footnotesize
\begin{tabular}{|c|c|c|} \hline
Problem name & Inputs & Outputs \\ \hline
\multirow{ 3}{*}{R$r$-PSDLS} & Data matrix ($D$) & Average computing time in seconds ($t$)\\
& Target matrix ($T$) & Average error value ($E$)\\
& $r$ &\\ \hline
\multirow{ 2}{*}{PSDLS} & $D$ & $t$\\
& $T$ & $E$\\ \hline
\multirow{ 3}{*}{MR} & $D$ & $t$\\
& $T$ & Average value of $\rank(X)$ ($r$)\\
& Error bound ($e$) & \\ \hline
\multirow{ 3}{*}{CM} & $T$ & Corellation matrix $X$ \\
& $P$ & Average standard deviation value of error matrix ($Std$)\\
& $Q$ &\\ \hline
\end{tabular}
\end{center}
\end{table}
\normalsize
For a given input size, ten random inputs are generated and the average value of outputs on the ten problems are reported. In Section 5.1, the numerical results corresponding to R$r$-PSDTLS and the one due to solving the linear problem (\ref{17}) by each of the three possible methods are presented. We refer to R$r$-PSDTLS using conjugate gradient method in the operator form to solve the linear problem as R$r$-PSDTLS-CG-O, R$r$-PSDTLS using conjugate gradient method after converting the linear problem to a linear system of equations as R$r$-PSDTLS-CG-L and R$r$-PSDTLS using GMRES in the operator form to solve the linear problem as R$r$-PSDTLS-GMRES-O. Considering these numerical results, we can make the following observations:\\\\
(1) R$r$-PSDTLS-CG-L generates solutions for which the orthogonality constraint, $\indent Y^TY=I$, is exactly satisfied; however, due to the high computing cost it is not $\indent$practical for large problems.\\
(2) R$r$-PSDTLS-GMRES-O computes the solution in a less computing time than R$r$-$\indent$PSDTLS-CG-O and R$r$-PSDTLS-CG-L.\\

Some numerical results are also reported in Section 5.1 to compare our proposed approach for solving the positive semi-definite least squares (PSDLS) problem by two existing methods, the Interior point method
(PSDLS-IntP in \cite{18}), and the path following method described by Toh (PSDLS-PFToh in \cite{30}).\\

Numerical results corresponding to special problems mentioned in Section 4 are reported in Section 5.2. There, numerical results obtained by our proposed algorithm for solving the minimum rank (MR) problem (MR-PSDTLS) and two other methods are reported. These two methods are MR-Toh \cite{31} and MR-Recht \cite{10,16}. Finally, the numerical results corresponding to our proposed method (CM-PSDTLS) and two other methods, CM-IntP and CM-Sun \cite{30}, for solving the correlation matrix (CM) problem are also reported in Section 5.2. In all the tables, the headings $m$ and $n$ correspond to the matrix size (the size of both data and target matrices) and $r$ gives the unknown matrix rank and the columns with headings $t$ and $E$ respectively contain the average computing time and error value.\\

In summary, numerical results confirm the effectiveness of PSDTLS in producing more accurate solutions with lower standard deviation values in less times for solving the rank $r$ positive semi-definite total least squares problems. The reported results show that our proposed methods solve the positive semi-definite total least squares problem and the minimum rank problem with a smaller value of error, while being more efficient. Also, the presented results confirm the efficiency of our proposed method in solving the correlation matrix problem.\\
\subsection{\textbf{Positive Semi-definite Total Least Squares Problem}}
Here, we report the numerical results for solving rank $r$ positive semi-definite total least squares problems and the general positive semi-definite total least squares problem respectively.\\\\
\textbf{Rank $r$ Positive Semi-definite Total Least Squares Problem.}
In Table \ref{t4}, the average computing time, $t$, and the average error value,
\[E=\tr {(DX-T)}^T(D-TX^{\dag})U_rU_r^T,\]
are reported for PSDTLS-CG-O. The fourth column gives norm of the error corresponding to the orthogonality constraint, $\delta=\|Y^TY-I\|$.

\begin{table}[htbp]
\caption{The average computing times and the average error values for PSDTLS-CG-O.}
\label{t4}
\begin{center} \footnotesize
\begin{tabular}{|c|c|c|c|c|c|} \hline
$m$ & $n$ & $r$ & $\delta$ & $t$ & $E$ \\ \hline
20 & 10 & 5 & 1.3421E-007 & 9.9864E-004 & 1.9459E+001\\
100 & 20 & 10 & 6.7435E-006 & 8.1695E-004 & 2.0093E+002\\
100 & 50 & 50 & 5.3193E-004 & 3.3691E-003 & 8.2009E+002\\
200 & 100 & 50 & $\ast$ & $\ast$ & $\ast$\\
400 & 300 & 200 & $\ast$ & $\ast$ & $\ast$\\
1000 & 500 & 200 & $\ast$ & $\ast$ & $\ast$\\\hline
\multicolumn{6}{c}{$\ast$: Out of memory.}\\
\end{tabular}
\end{center}
\end{table}
\normalsize
In tables \ref{t5} and \ref{t6}, we report the results for PSDTLS-GMRES-O and PSDTLS-CG-L respectively.\\
\newpage
\begin{table}[htbp]
\caption{The average computing times and the average error values for PSDTLS-GMRES-O.}
\label{t5}
\begin{center} \footnotesize
\begin{tabular}{|c|c|c|c|c|c|} \hline
 $m$ & $n$ & $r$ & $\delta$ & $t$ & $E$ \\ \hline
20 & 10 & 5 & 7.4924E-008 & 7.1055E-004 & 2.2493E+001\\
100 & 20 & 10 & 1.7123E-006 & 1.1905E-003 & 2.0953E+002\\
100 & 50 & 50 & 4.3528E-005 & 3.2913E-003 & 1.4490E+003\\
200 & 100 & 50 & 2.5649E-004 & 1.3800E-002 & 2.6041E+005\\
400 & 300 & 200 & 3.7246E-003 & 1.5974E-001 & 1.0264E+006\\
1000 & 500 & 200 & 1.0021E-002 & 3.6251E-001 & 4.5691E+006\\\hline
\end{tabular}
\end{center}
\end{table}
\normalsize
\begin{table}[htbp]
\caption{The average computing times and the average error values for PSDTLS-CG-L.}
\label{t6}
\begin{center} \footnotesize
\begin{tabular}{|c|c|c|c|c|c|} \hline
$m$ & $n$ & $r$ & $\delta$ & $t$ &$E$ \\ \hline
20 & 10 & 5 & 0 & 1.0904E-003 & 1.7025E+001\\
100 & 20 & 10 & 1.0032E-009 & 1.1484E-003 & 1.8363E+002\\
100 & 50 & 50 & $\ast$ & $\ast$ & $\ast$\\
200 & 100 & 50 & $\ast$ & $\ast$ & $\ast$\\
400 & 300 & 200 & $\ast$ & $\ast$ & $\ast$\\
1000 & 500 & 200 & $\ast$ & $\ast$ & $\ast$\\\hline
\end{tabular}
\end{center}
\end{table}

\normalsize

In Figure \ref{f1}, the Dolan-Mor\'{e} time profile is presented to compare the computing times by PSDTLS-CG-L, PSDTLS-GMRES-O and PSDTLS-CG-L for solving the PSDLS problems. Generating enough test problems is necessary for producing an illustrative performance profile. We generated 300 test problems (50 problems for each matrix size). The presented time profile shows that PSDTLS-GMRES-O needs much less computing time than the other methods.

\begin{figure}[ht!]
       \includegraphics[width=0.7\textwidth]{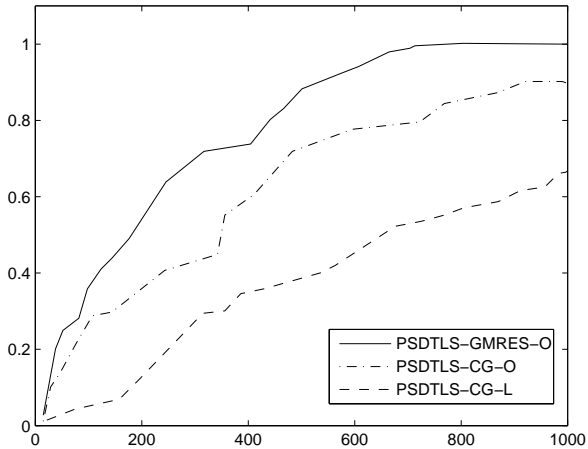}
     \caption{The Dolan-Mor\'{e} performance profile
(comparing the computing times by PSDTLS-CG-L, PSDTLS-GMRES-O and PSDTLS-CG-L).}
     \label{f1}
\end{figure}

$\hspace{-0.6cm}$\textbf{Positive Semi-definite Total Least Squares Problem.}
Here, the numerical results for solving the general positive semi-definite total least squares problems are reported. In tables \ref{t7}, \ref{t8} and \ref{t9}, the average computing times, $t$, and the average error values, $E$, for solving the PSDLS problem using PSDTLS-CG-O, PSDTLS-CG-L and PSDTLS-GMRES-O are respectively reported. The third column represents the average value of $\delta=\|Y^TY-I\|$.

\begin{table}[htbp]
\caption{The average computing times and the average error values for PSDTLS-CG-O.}
\label{t7}
\begin{center} \footnotesize
\begin{tabular}{|c|c|c|c|c|} \hline
$m$ & $n$  & $\delta$ & $t$ & $E$ \\ \hline
20 & 10 &  1.0347E-007 & 5.0208E-004 & 4.7606E+000 \\
100 & 20 & 6.4238E-006 & 9.8582E-004 & 2.1158E+001\\
100 & 50 & 1.0046E-004 & 4.3211E-003 & 2.5376E+001\\
200 & 100 & $\ast$ & $\ast$ & $\ast$\\
400 & 300 & $\ast$ & $\ast$ & $\ast$\\ \hline
\end{tabular}
\end{center}
\end{table}
\normalsize
\begin{table}[htbp]
\caption{The average computing times and the average error values for PSDTLS-GMRES-O.}
\label{t8}
\begin{center} \footnotesize
\begin{tabular}{|c|c|c|c|c|} \hline
$m$ & $n$ & $\delta$ & $t$ & $E$ \\ \hline
20 & 10 & 8.1282E-009 & 7.3702E-004 & 3.0053E+000\\
100 & 20 & 5.7631E-008 & 1.1974E-003 & 1.9550E+001\\
100 & 50 & 1.3214E-007 & 4.9063E-003 & 2.5928E+001\\
200 & 100 & 6.8472E-006 & 1.4990E-002 & 5.7797E+001\\
400 & 300 & 2.1064E-004 & 1.8509E-001 & 7.5146E+001\\ \hline
\end{tabular}
\end{center}
\end{table}
\normalsize

\begin{table}[htbp]
\caption{The average computing times and the average error values for PSDTLS-CG-L.}
\label{t9}
\begin{center} \footnotesize
\begin{tabular}{|c|c|c|c|c|} \hline
$m$ & $n$ & $\delta$ & $t$ & $E$ \\ \hline
20 & 10 & 0 & 8.9295E-004 & 3.0689E+000\\
100 & 20 & 0 & 1.6088E-003 & 1.6594E+001\\
100 & 50 & $\ast$ & $\ast$ & $\ast$\\
200 & 100 & $\ast$ & $\ast$ & $\ast$\\
400 & 300 & $\ast$ & $\ast$ & $\ast$\\ \hline
\end{tabular}
\end{center}
\end{table}
\normalsize
 Considering the reported results in tables \ref{t7}, \ref{t8} and \ref{t9}, PSDTLS-CG-O and PSDTLS-GMRES-O perform approximately the same on small problems; however, for large problems, PSDTLS-GMRES-O outperforms PSDTLS-CG-O in almost all the test problems. Thus, we report the results obtained by PSDTLS-GMRES-O in comparisons with other methods.\\

 The average computing times for solving the positive semi-definite total least squares problem using PSDTLS, PSDLS-IntP and PSDLS-PFToh are reported in Table \ref{t10}. The third column presents the values of $TOL$ for PSDLS-IntP and PSDLS-PFToh.\\
\begin{table}[htbp]
\caption{The average computing times for solving PSDLS Using PSDTLS, PSDLS-IntP and PSDLS-PFToh.}
\label{t10}
\begin{center} \footnotesize
\begin{tabular}{|c|c|c|c|c|c|} \hline
{\multirow{2}{*}{$m$}} & {\multirow{2}{*}{$n$}} &  $TOL$ & \multicolumn{3}{c|}{$t$} \\ \cline{4-6}
{} & {}  & (IntP/PFToh) & PSDTLS  &  IntP & PFToh \\ \hline
20 & 10  & 1.0000E-006 & 5.0208E-004 & 3.1785E-003 & 1.8652E-003\\
100 & 20  & 1.0000E-006 & 9.8582E-004 & 6.1142E-002 & 7.6139E-003\\
100 & 50  & 1.0000E-005 & 4.0462E-003 & 9.2783E-001 & 5.2841E-002\\
200 & 100  & 1.0000E-005 &  1.3841E-002 & $\ast$ & 1.5973E-001\\
400 & 300  & 1.0000E-004 &  1.6940E-001 & $\ast$ &  $\ast$\\ \hline
\end{tabular}
\end{center}
\end{table}

\normalsize
Similarly, the average error values, $E$, are given in Table \ref{t11} for PSDLS-IntP, PSDLS-PFToh and PSDTLS.\\
\newpage
\begin{table}[htbp]
\caption{The average error values for solving PSDLS-Using PSDTLS, PSDLS-IntP and PSDLS-PFToh.}
\label{t11}
\begin{center} \footnotesize
\begin{tabular}{|c|c|c|c|c|c|} \hline
{\multirow{2}{*}{$m$}} & {\multirow{2}{*}{$n$}} &  $TOL$ & \multicolumn{3}{c|}{$t$} \\ \cline{4-6}
{} & {}  & (IntP/PFToh) & PSDTLS  &  IntP & PFToh \\ \hline
20 & 10  & 1.0000E-006 & 2.7168E+000 & 3.7791E+000 & 5.1176E+000\\
100 & 20  & 1.0000E-006 & 9.5355E+000 & 1.2956E+001 & 9.6873E+000\\
100 & 50  & 1.0000E-005 &  1.0158E+001 & 1.9162E+001 & 1.4855E+001\\
200 & 100  & 1.0000E-005 &  2.1574E+001 & $\ast$ & 4.8546E+002\\
400 & 300  & 1.0000E-004 &  1.0976E+002 & $\ast$ & $\ast$\\ \hline
\end{tabular}
\end{center}
\end{table}
\normalsize

The corresponding Dolan-Mor\'{e} time profile is shown in Figure \ref{f2} to compare the computing times for solving the PSDLS problem. We generated 500 test problems (100 for each matrix size) to provide an illustrative time profile. The results confirm that our proposed algorithm for solving positive semi-definite least squares problem performs much faster than the other methods.\\

\begin{figure}[ht!]
       \includegraphics[width=0.7\textwidth]{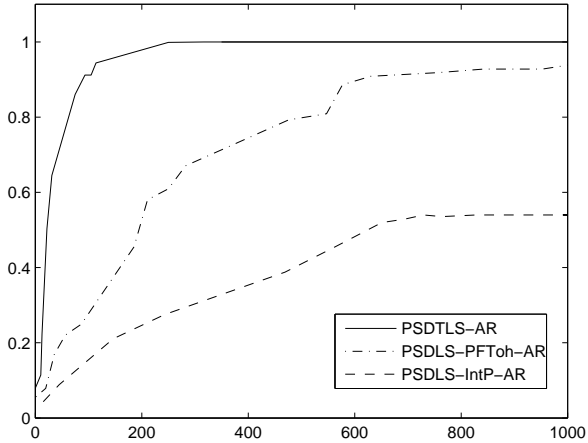}
     \caption{The Dolan-Mor\'{e} performance profile
(comparing the computing times by PSDTLS, PSDLS-IntP and PSDLS-PFToh).}
     \label{f2}
\end{figure}

\subsection{\textbf{Special Problems}}
Here, we report numerical results corresponding to the special problems mentioned in Section 4.
In Table \ref{t12}, the average computing times for solving the minimum rank (MR) problem are reported. The third column gives the value of error bound, $e$, in the MR problem.

\begin{table}[htbp]
\caption{The average computing times for solving the MR problem-Using MR-PSDTLS-CG-O, MR-PSDTLS-GMRES-O and MR-PSDTLS-CG-L methods.}
\label{t12}
\begin{center} \footnotesize
\begin{tabular}{|c|c|c|c|c|c|c|c|c|} \hline
{\multirow{2}{*}{$m$}} & {\multirow{2}{*}{$n$}} & {\multirow{2}{*}{$e$}} & \multicolumn{3}{c|}{$t$} & \multicolumn{3}{c|}{$r$} \\ \cline{4-9}
{} & {} & {} & CG-O  &  GMRES-O & CG-L & CG-O  &  GMRES-O & CG-L \\ \hline
20 & 10  & 10 & 1.0000E-003 & 1.1000E-003 & 1.0000E-003 & 3 & 3 & 4\\
100 & 20  & 300 & 9.5430E-004 & 1.8000E-003 & 1.9000E-003 & 17 & 14 & 18 \\
100 & 50  & 500 &  4.7000E-003 & 4.3000E-003 & 4.2000E-003 & 22 & 16 & 24\\
200 & 100  & 1000 &  1.4000E-002 & 1.8400E-002 & $\ast$ & 19 & 19 & $\ast$\\
400 & 300  & 3000 &  8.0024E-001 & 1.8700E-001 & $\ast$ & 67 & 66 & $\ast$\\ \hline
\end{tabular}
\end{center}
\end{table}
\normalsize

The average computing time needed by MR-PSDTLS, MR-Toh and MR-Recht and the average resulting rank, $r$, for solving the MR problem are reported in Table \ref{t13}.

\begin{table}[htbp]
\caption{The average computing times and the average error values for solving the MR problem-Using MR-PSDTLS, MR-Toh and MR-Recht methods.}
\label{t13}
\begin{center} \footnotesize
\begin{tabular}{|c|c|c|c|c|c|c|c|c|} \hline
{\multirow{3}{*}{$m$}} & {\multirow{3}{*}{$n$}} & {\multirow{3}{*}{$e$}} & \multicolumn{3}{c|}{$t$} & \multicolumn{3}{c|}{Rank} \\ \cline{4-9}
{} & {} & {} & MR-  &  MR- & MR- & MR-  &  MR- & MR- \\
{} & {} & {} & PSDTLS  &  Toh & Recht & PSDTLS  &  Toh & Recht \\ \hline
20 & 10  & 10 & 1.0000E-003 & 9.0300E-003 & 7.8100E-003 & 3 & 3 & 4\\
100 & 20  & 300 & 9.5430E-004 & 1.2300E-002 & 3.1100E-002 & 14 & 12 & 11 \\
100 & 50  & 500 &  4.2000E-003 & 2.1800E-002 & 9.8800E-003 & 16 & 13 & 13\\
200 & 100  & 1000 &  1.4000E-002 & 8.9700E-002 & 9.1200E-002 & 19 & 18 & 19\\
400 & 300  & 3000 &  1.8700E-001 & $\ast$ & 3.1060E+000& 67 & $\ast$ & 65\\ \hline
\end{tabular}
\end{center}
\end{table}
\normalsize
An illustrative comparison is also provided in Figure \ref{f3} based on the corresponding Dolan-Mor\'{e} time profile. We generated 500 test problems to provide the time profile. The presented Dolan-Mor\'{e} time profile shows that our proposed algorithm for solving minimum rank problem needs less computing time than the other two methods.
\newpage
\begin{figure}[ht!]
       \includegraphics[width=0.7\textwidth]{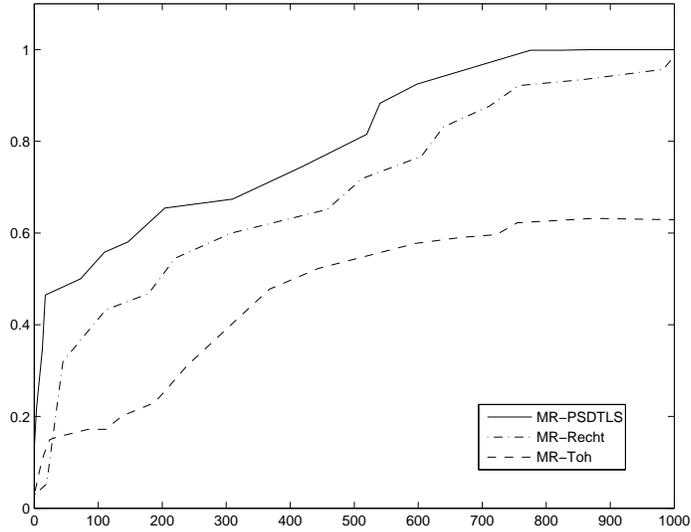}
     \caption{The Dolan-Mor\'{e} performance profile
(comparing the computing times by MR-PSDTLS, MR-Toh and MR-Recht).}
     \label{f3}
\end{figure}
In Table \ref{t14}, the average computing times are reported for computing the correlation matrices by our proposed method (CM-PSDTLS).

\begin{table}[htbp]
\caption{The average computing times for computing the correlation matrix-Using CM-PSDTLS-CG-O, CM-PSDTLS-GMRES-O and CM-PSDTLS-CG-L.}
\label{t14}
\begin{center} \footnotesize
\begin{tabular}{|c|c|c|c|c|} \hline
{\multirow{2}{*}{$m$}} & {\multirow{2}{*}{$n$}} & \multicolumn{3}{c|}{$t$} \\ \cline{3-5}
{} & {} & CG-O  &  GMRES-O & CG-L  \\ \hline
20 & 10   & 1.1000E-003 & 1.3000E-003 & 1.3000E-003 \\
100 & 20  & 2.1000E-003 & 1.6000E-003 & 3.7000E-003 \\
100 & 50  &  3.5000E-003 & 6.7000E-003 & $\ast$ \\
200 & 100 &  $\ast$ & 1.6400E-002 & $\ast$ \\
400 & 300 &  $\ast$ & 2.9800E-002 & $\ast$ \\ \hline
\end{tabular}
\end{center}
\end{table}
\normalsize
Similarly, in Table \ref{t15} the average value of standard deviation $Std$ for computing the correlation matrix is reported.

\begin{table}[htbp]
\caption{The average standard deviation values of error matrix for computing the correlation matrices, using CM-PSDTLS-CG-O, CM-PSDTLS-GMRES-O and CM-PSDTLS-CG-L .}
\label{t15}
\begin{center} \footnotesize
\begin{tabular}{|c|c|c|c|c|} \hline
{\multirow{2}{*}{$m$}} & {\multirow{2}{*}{$n$}} & \multicolumn{3}{c|}{$Std$} \\ \cline{3-5}
{} & {} & CG-O  &  GMRES-O & CG-L  \\ \hline
20 & 10   & 3.0410E-001 & 3.0280E-001 & 4.0030E-001 \\
100 & 20  & 2.9940E-001 & 2.9140E-001 & 3.3810E-001 \\
100 & 50  &  2.9310E-001 & 2.8870E-001 & $\ast$ \\
200 & 100 &  $\ast$ & 2.8840E-001 & $\ast$ \\
400 & 300 &  $\ast$ & 2.8950E-001 & $\ast$ \\ \hline
\end{tabular}
\end{center}
\end{table}
\normalsize
The average computing times for CM-PSDTLS, CM-IntP and CM-Sun are reported in Table \ref{t16}.
\newpage

\begin{table}[htbp]
\caption{The average computing times for computing the correlation matrix-Using CM-PSDTLS, CM-IntP and CM-Sun.}
\label{t16}
\begin{center} \footnotesize
\begin{tabular}{|c|c|c|c|c|} \hline
{\multirow{2}{*}{$m$}} & {\multirow{2}{*}{$n$}} & \multicolumn{3}{c|}{$t$} \\ \cline{3-5}
{} & {} &  CM-PSDTLS  &  CM-IntP & CM-Sun \\ \hline
20 & 10   & 1.1000E-003 & 2.4100E-002 & 9.6640E-003 \\
100 & 20  & 1.6000E-003 & 1.6140E-001 & 5.4100E-002 \\
100 & 50  &  3.5000E-003 & 7.4941E+000 & 4.7810E-001 \\
200 & 100 &  1.6400E-002 & $\ast$ & 2.9871E+000 \\
400 & 300 &   2.9800E-002 & $\ast$ & $\ast$\\ \hline
\end{tabular}
\end{center}
\end{table}
\normalsize

In Figure \ref{f4}, the Dolan-Mor\'{e} time profile is presented to compare the needed computing times by CM-PSDTLS, CM-IntP and CM-Sun to solve the correlation matrix problem. Here, we generated 250 test problems (50 for each matrix size). The presented time profile confirms that our proposed algorithm computes a correlation matrix much faster than the other methods.

\begin{figure}[ht!]
       \includegraphics[width=0.7\textwidth]{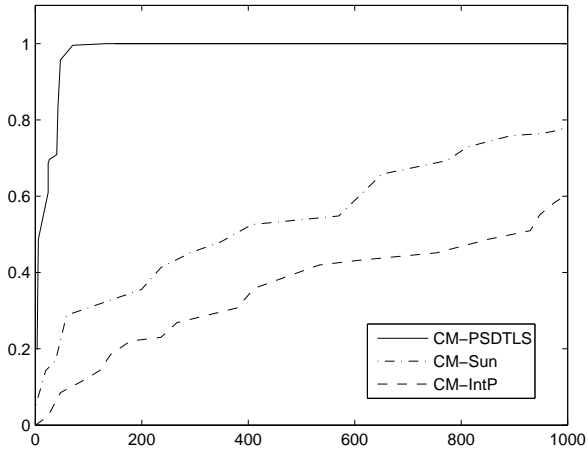}
     \caption{The Dolan-Mor\'{e} performance profile
(comparing the computing times by CM-PSDTLS, CM-IntP and CM-Sun).}
     \label{f4}
\end{figure}
In Table \ref{t17}, the average values of the standard deviation, $Std$, for computing the correlation matrix are reported.

\begin{table}[htbp]
\caption{The average standard deviation values of error matrix for computing the correlation matrices, using CM-PSDTLS, CM-IntP and CM-Sun.}
\label{t17}
\begin{center} \footnotesize
\begin{tabular}{|c|c|c|c|c|} \hline
{\multirow{2}{*}{$m$}} & {\multirow{2}{*}{$n$}} & \multicolumn{3}{c|}{$Std$} \\ \cline{3-5}
{} & {} &  CM-PSDTLS  &  CM-IntP & CM-Sun \\ \hline
20 & 10   & 3.0410E-001 & 4.0404E+003 & 1.2957E+002 \\
100 & 20  & 2.9940E-001 & 1.7266E+004 & 6.3595E+004 \\
100 & 50  &  2.9310E-001 & 8.2909E+005 & 3.8741E+006 \\
200 & 100 &  $\ast$ & 2.8840E-001 & 1.2389E+007 \\
400 & 300 &  $\ast$ & 2.8950E-001 & $\ast$ \\ \hline
\end{tabular}
\end{center}
\end{table}
\normalsize
Considering the numerical results reported in this section, we summarize our observations:\\\\
(1) A newly defined problem, R$r$-PSDLS, was considered and an efficient algorithm\\$\indent$was proposed for its solution.\\
(2) Although our proposed algorithm for solving the PSDLS problem, PSDTLS, ap-\\$\indent$plies R$r$-PSDTLS, for $r=1,\ldots,n$, in search for the solution, it appears to be $\indent$more efficient than PSDLS-IntP and PSDLS-PFToh.\\
(3) In contrast with other available methods, our use of total formulation for solving\\$\indent$the PSDLS problem to consider error in both data and target matrices turns to be $\indent$practically effective to produce more meaningful results.\\
(4) The proposed method for solving the PSDLS problem, PSDTLS, is more efficient\\$\indent$than the other methods.\\
(5) The proposed method for solving the minimum rank problem, MR-PSDTLS, is $\indent$also more efficient than MR-Toh and MR-Recht.\\
(6) The proposed method for computing the correlation matrix, CM-PSDTLS, shows\\$\indent$to be more efficient and robust in computing a correlation matrix with a lower $\indent$value of standard deviation of error in $T$ as compared to CM-IntP and CM-Sun.\\

\section{Concluding Remarks}
We proposed a new approach to solve positive semi-definite total least squares (PSDLS) problems. Consideration of our proposed error estimate for both data and target matrices admitted a more realistic problem formulation. We first considered a newly defined given rank positive semi-definite total least squares (R$r$-PSDTLS) problem and presented an at least quadratically convergent algorithm for its solution. Numerical results confirmed the effectiveness of our approach to compute solutions of R$r$-PSDLS problems in less computing time than the interior point method and the path following algorithm. We then showed how to apply R$r$-PSDTLS to solve the general PSDLS problem. Based on the reported numerical results, our method for solving the PSDLS problem also showed to be more efficient than the interior point method and the path following algorithm. An specifically effective approach was also described to solve the rank $1$ positive semi-definite total least squares problem, R$1$-PSDTLS. In addition, we noted that R$r$-PSDTLS can be applied to other problems arising in control and financial modeling: the minimum rank (MR) problem and correlation matrix computation. Using the Dolan-Mor\'{e} performance profiles, we showed our proposed method for solving the MR problem to be more efficient than a path following algorithm and a semi-definite programming approach for solving the MR problem. Furthermore, in computing the correlation matrix, numerical results showed lower standard deviation of error as compared to the interior point method and semi-definite programming approach.\\\\

\textsc{\textbf{ACKNOWLEDGEMENT.}} The authors thank Research Council of Sharif University of Technology for supporting this work.\\

\end{document}